\documentclass[12pt]{article}
\usepackage{graphicx}
\usepackage{amsfonts }
\usepackage{amsmath}
\usepackage{fullpage}
\usepackage[applemac]{inputenc}
\usepackage{enumerate}
\usepackage{amssymb}
\usepackage{amsthm}
\usepackage{setspace}
\usepackage{cite}
\usepackage{fancyhdr}
\usepackage{anysize}
\marginsize{3cm}{1.5cm}{3cm}{7cm}
\theoremstyle{definition}
\newtheorem{teo}{Theorem}[section]
\newtheorem{cor}[teo]{Corollary}
\newtheorem{lem}[teo]{Lemma}
\newtheorem{conj}[teo]{Conjecture}
\newtheorem{rem}[teo]{Remark}

\newtheorem{defi}[teo]{Definition}

\newtheorem{que}[teo]{Question}
\begin{document}
\newcommand{\A}{{\cal A}}
\renewcommand{\O}{{\cal O}}
\newcommand{\B}{{\cal B}}
\newcommand{\M}{{\cal M}}
\newcommand{\C}{{\cal C}}
\newcommand{\F}{{\mathbb F}}
\newcommand{\Prob}{{\mathbb P}}
\newcommand{\I}{{\cal I}}
\newcommand{\R}{{\mathbb R}}
\newcommand{\Z}{{\mathbb Z}}
\newcommand{\X}{{\cal X}}
\newcommand{\Y}{{\cal Y}}
\newcommand{\ZZ}{{\cal Z}}

\newcommand{\conv}{{\mathrm conv}}
\newcommand{\ehr}{\text{Ehr}}
\renewcommand{\span}{\text {span}}

\renewcommand{\a}{{\mathbf a}}
\renewcommand{\b}{{\mathbf b}}
\newcommand{\e}{{\mathbf e}}
\renewcommand{\v}{{\mathbf v}}
\newcommand{\x}{{\mathbf x}}

\title{ \bf The Maxflow problem and a generalization to simplicial complexes}
\author{Fabi\'an Latorre}
\date{2012}

\maketitle
\newpage 
\noindent  {\bf \large Acknowledgements}

\vspace{2cm}

Gracias a la cigue\~na y quienes, no siendo su decisi\'on, han tenido que vivir en mi tiempo y beber junto a m\'{i} y junto a Baco. Gracias por ser mis contempor\'aneos, y que la casualidad nos haya llevado a conocernos y tal vez, ver un poco m\'as que un aut\'omata el uno en el otro. 

Gracias a mi asesor Mauricio Velasco a quien ha dedicado gran parte de su tiempo a gu\'{i}ar este proyecto, a mis padres y hermanos.
\newpage

\vspace{2cm}

 \tableofcontents

\newpage

\vspace{2cm}

 \section{Introduction}
 
 The problem of Maxflow was formulated by T.E. Harris in 1954 while studying the Soviet Union's railway network, under a military research program financed by RAND, Research and Development corporation. The research remain classified until 1999. The Maxflow problem is defined on a \emph{network} which is a directed graph together with a real positive capacity function defined on the set of edges of the graph and two vertices $s,t$ called the \emph{source} and the \emph{sink}. A flow is another function of this type that respects capacity constraints and a Kirchoff's law type restriction on each vertex except the source and sink. The net flow of a flow is defined as the amount of flow leaving the source. The problem of maxflow is to find a flow with maximum net flow on a given network. In the first section we will define clearly such concepts and present basic results in the subject.

Throughout the second section, we will present three different algorithms for the solution of Maxflow. In 1956 L. Ford and D. Fulkerson devised the first known algorithm that solves the problem in polynomial time. The algorithm works starting with the zero flow and finding paths from source to sink where flow can be augmented preserving the flow and capacity restrictions. We then analyze a more efficient algorithm developed by A. Goldberg and E. Tarjan in 1988. This algorithm works in a different fashion starting with a \emph{preflow}, a function saturating edges adjacent to the source, and then pushing excess of flow to vertices estimated to be closer to the sink. At the end of the algorithm the preflow becomes a flow and in fact, a maximum flow.  Finally we describe Dorit Hochbaum's \emph{pseudoflow} algorithm, which is the most efficient algorithm known to day that solves the Maxflow problem.

In the third section we show the usefulness of this subject and present three applications of the theory of Network flow. First we show how well-known theorems in combinatorics such as the Hall's Marriage theorem can be proven using Maxflow results. We then show how to find a set of maximal chains in a poset with certain properties using the results in the previous sections. Finally we describe an algorithm for \emph{image segmentation}. an important subject in computer vision, that relies on the relation between a maximum flow and a \emph{minimum cut}. 

The problem of Maxflow is a widely developed subject in modern mathematics.  Efficient algorithms exist to solve this problem, that is why a good generalization may permit these algorithms to be understood as a particular instance of solutions in a wider class of problems. In the last section we suggest a generalization in the context of simplicial complexes, that reduces to the problem of Maxflow in graphs, when we consider a graph as a simplicial complex of dimension 1. 

\newpage

\section{Preliminaries}
\vspace{1cm}
\subsection{Flow in a network}\label{flowian}

There are many equivalent ways to define the objects needed to state our problem. We will work with the following:

\begin{defi}
A \emph{network} is a pair $(G,c)$ such that
\begin{enumerate}[i)]
\item $G=(V,E)$ is a finite simple directed graph
\item $V=V' \cup \{s,t\}$
\item for any $v \in V$, $(v,s) \not \in E$ and $(t,v) \not \in E$.
\item $c: E \longrightarrow \R_+ $
\end{enumerate}
\end{defi}

We call $s$ and $t$ the \emph{source} and the \emph{sink} respectively. Condition iii) means that there are no edges \emph{into} the source, and no edges \emph{out} of the sink.
\begin{defi}
For any simple directed graph we define the \emph{incidence function}\\ $\phi: V \times E \longrightarrow \{-1,0,1\} $ as follows:

\begin{displaymath}
   \phi (v,e) = \left\{
     \begin{array}{rl}
       1 & : e = (v,x) \\
       -1 & : e = (x,v) \\
       0 &:  \mbox{else}
     \end{array}
   \right.
\end{displaymath} 

\end{defi}

\begin{defi}\label{flow}
A flow on a network $(G,c)$ is a function $f:E \longrightarrow \R_+ $ such that:
\begin{enumerate}[i)]

\item for any edge $e$, $f(e) \leq c(e) $

\item for any vertex $v \not = s,t $ \emph{conservation of flow} holds:  $$\sum_{e \in E} f(e)\phi(v,e) = 0$$  \end{enumerate}

\end{defi}

\begin{defi}
For a flow $f$ on a network $(G,c)$, we define the \emph{net flow}:
$$|f|:=\sum_{e \in E} f(e)\phi(s,e) $$
$|f|$ is the total amount flowing out of the source.
\end{defi}

\vspace{1cm}
\subsection{The problem of MAXFLOW} \label{lp}

Given a network $(G,c)$ the MAXFLOW problem is to find a flow $f$ of maximum net flow.

\begin{teo}\label{existence} For any network $(G,c)$ there exists a flow $f$ of maximum net flow.
\begin{proof}
Let $m:[|E|] \rightarrow E $ be an enumeration of the edges of $G$. Let $ {\cal F} $ be the set of feasible flows on $(G,c)$. The map $ \psi : {\cal F } \rightarrow A \subset \R^{|E|} $, $f \mapsto [f(m(i))]_i $ is a bijection between ${\cal F}$ and a subset of $\R^{|E|}$. Let $ {\cal F}^*$ be the image of $\cal F$ under this map. From definition \ref{flow}, the edge capacity constraints imply that $\cal F^*$ is bound, and the flow conservation constraints imply that it is closed, hence $\cal F^*$ is compact. The map $ f^* \in {\cal F^*} \mapsto |\psi^{-1}( f^*)| \in \R $ is a linear map, hence continuous. As it is defined on a compact set, it achieves a maximum value, say at $f_{\mbox{max}}$. $\psi^{-1} (f_{\mbox{max}}) $ is then a flow of maximum value.\\
\end{proof} 

\end{teo}

The previous theorem shows that, in fact, MAXFLOW is a linear programming problem, the most important results of which can be proved with LP theory. We discuss this formulation in detail in what follows.

\begin{defi}\label{incidencematrix}
Let $G=(V,E)$ be a simple directed graph, $\phi$ its incidence function and $(G,c)$ a network. Let $n:=|V| $, $m:=|E|$, $v:[n] \rightarrow V$ be an enumeration of the vertices and $e:[m] \rightarrow E $  be an enumeration of the edges. We define the \emph{incidence matrix} $\Phi_{v,e}$with respect to the enumerations $v,e$ as 
$$
\Phi_{v,e}(i,j)= \phi(v(i), e(j))
$$
\end{defi}

\begin{figure}[h!]
\begin{center}
\includegraphics[scale=0.4]{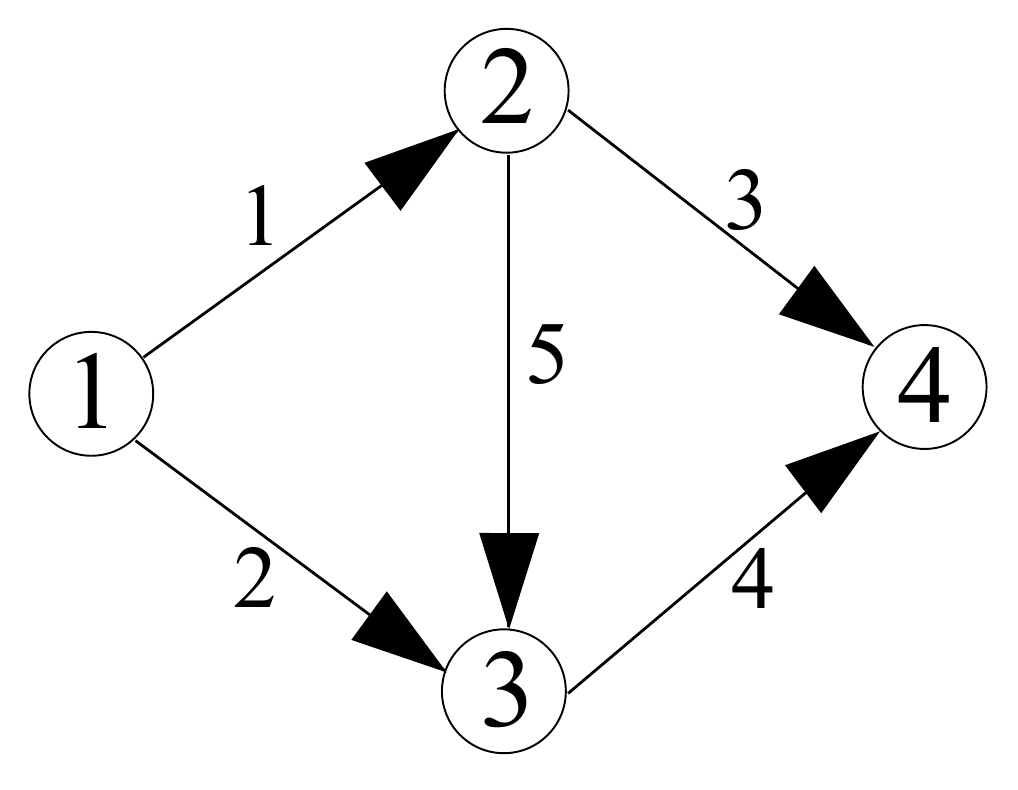} 
\caption{a graph $G_1$}
\end{center}
\end{figure} 

\begin{table}[h!]
  \begin{center}
 $$ \left [  \begin{tabular}{ c c c c c }
   
    1 & 1 & 0 &0 &0 \\
    -1 & 0 & 1 & 0 &1 \\
    0 & -1 & 0 & 1 & -1 \\
    0 & 0 & -1 & -1 & 0
  
    \end{tabular} \right ] $$
  \end{center}
  \caption{Incidence matrix of $G_1$}
\end{table}

From now on, we suppose a network $(G,c)$ has fixed enumerations $v,e$, of vertices and edges. We take $n$ as the number of vertices, $m$ as the number of edges and suppose that $v(1)=s$, $v(n)=t$, then we refer simply to the incidence matrix as $\Phi$.

\begin{lem} Given a network $(G,c)$, the problem of MAXFLOW is equivalent to the following LP problem:
\begin{displaymath}
\begin{array}{ccccccc}
{\displaystyle \mathop{\max}_{x \in \R^m} \Phi_{1,j} \cdot  x }& ;& \Phi^* x = 0 &;& I_m x \leq c^* & ;& x \geq 0
\end{array}
\end{displaymath}
where $\Phi_{1,j}$ is the first row vector of the matrix $\Phi$, $\Phi^*$ is the matrix that results from $\Phi$ by deleting the first and last rows, $I_m$ is the identity matrix of size $m$ and $c^*:=[c(e(i))]_i $ is the vector of edge capacities.
\begin{proof}
Follows from the definition \ref{flow} and the proof of theorem \ref{existence}.\\
\end{proof}

\end{lem}

\begin{defi}\label{dual}
For a linear program (called the primal problem)
\begin{displaymath}
\begin{array}{ccccc}
{\displaystyle \mathop{\max}_{x} (c^T x) }& ;& A x \leq b & ; & x \geq 0
\end{array}
\end{displaymath}

the dual program is defined as 

\begin{displaymath}
\begin{array}{ccccc}
{\displaystyle \mathop{\min}_{y} (b^T  y) }& ;& A^T y \geq c & ; & y \geq 0
\end{array}
\end{displaymath}

\end{defi}

We will compute the dual program of MAXFLOW.

\begin{defi}\label{cut} Let $(G,c)$ be a network. A \emph{cut} is a partition of $V$ into two disjoint subsets $(S,\bar S) $ such that $s \in S$ and $t \in \bar S $. Let $E':= \{ e \in E : e=(s,\bar s) \, s \in S \,,\, \bar s \in \bar S \} $. We define the \emph{capacity} of the cut $(S,\bar S)$ 
$$
C(S, \bar S) = \sum_{e \in E'} c(e)
$$
$C(S, \bar S) $ is the sum of the capacities of the edges directed from $S$ to $\bar S $. We say an edge $e$ \emph{traverses} the cut if $e=(s,\bar s) \, s \in S \, \bar s \in \bar S$

\end{defi}

\begin{lem}\label{value off} For any cut $(S,S')$
$$
|f|=f(S,S')-f(S',S)
$$
where $f(X,Y)$ is the sum of the values of $f$ at the edges directed from $X$ to $Y$.
\begin{proof}
$$
|f|=\sum_{e \in E} f(e)\phi(s,e)=\sum_{v \in S}\sum_{e \in E} f(e) \phi(v,e)= \sum_{e \in E}\sum_{ v \in S} f(e) \phi(v,e)
$$
$$
= \underbrace{ \sum_{e:S\rightarrow S} \sum_{ v \in S} f(e) \phi(v,e)}_{=0} + \underbrace{ \sum_{e:S\rightarrow S'}\sum_{ v \in S} f(e) \phi(v,e)}_{=f(S,S')} + \underbrace{ \sum_{e: S' \rightarrow S} \sum_{ v \in S} f(e) \phi(v,e)}_{-f(S',S)}
$$
\end{proof}
\end{lem}

\begin{lem}\label{maxleqmin}
Given a network $(G,c)$, for any flow $f$ and any cut $(S, \bar S)$
$$
|f| \leq C(S, \bar S) 
$$
\end{lem}

The problem of MINCUT is to find a cut of $(G,c)$ of minimum capacity. Two of the most important results are the following

\begin{teo}\label{maxflowmincut}
MAXFLOW=MINCUT. This means the net flow of a maximal flow is equal to the capacity of a minimal cut.
\end{teo}

In order to prove these results, we will see that the dual program of MAXFLOW is a relaxation of the MINCUT problem and use the following:

\begin{teo}\label{weakduality} \emph{(weak duality)}. 
Let $x^*$ and $y^*$ be feasible solutions to a primal problem and its dual, respectively, then
$$
c^T x^* \leq b^T y^*
$$
\end{teo}
\begin{proof}
For $y_1,y_2 \geq 0 $ define the function

$$
g(y_1,y_2):= {\max_x}  c^Tx + y_1^T(b-Ax) + y_2^T x 
$$
Clearly, for any feasible $x^*$ and $y_1,y_2 \geq 0 $, $c^Tx^* \leq g(y_1,y_2)$.\\

Rearranging terms we have

$$
g(y_1, y_2)= {\max_x} (c^T - y_1^TA + y_2^T)x + y_1^T b
$$

where $y_1 \in \R^m$ and $y_2 \in \R^n $. Then we have
\begin{displaymath}
   g(y_1,y_2) = \left\{
     \begin{array}{rl}
      y_1^T b & : c^T - y_1^TA + y_2^T=0 \\
       \infty & : \mbox{else} 
     \end{array}
   \right.
\end{displaymath} 

Now minimizing over $y_1,y_2 \geq 0 $ we have, for any feasible $x^*$

$$
c^Tx^* \leq \min_{y_1,y_2 \geq 0 } g(y_1,y_2) 
$$

By the previous observation this is equivalent to $\min_{y_1,y_2\geq 0}  y_1^T b $ subject to $  c^T - y_1^TA + y_2^T=0 $. We see that   $  c^T - y_1^TA + y_2^T=0  \Rightarrow c^T - y_1^TA = -y_2^T \leq 0 $ and this is equivalent to a single variable $y \geq 0 $ such that $c^T - y^T A \leq 0 \Rightarrow c^T \leq y^T A$ and this is the dual problem, as we wanted to show.

\end{proof}

\begin{teo}\label{strongduality} \emph{(strong duality)}. 
If the primal problem has an optimal solution $x^*$, then the dual problem also has an optimal solution $y^*$ and
$$
c^T x^* = b^T y^*
$$
\end{teo}

To find the dual of our problem, we state it in standard form

\begin{equation}\label{primal}
\begin{array}{ccccc}
{\displaystyle \mathop{\max}_{x \in \R^m} \Phi_{1,j} \cdot  x }& ; & \left [ \begin{array}{c} \Phi^* \\ - \Phi^* \\ I_m
\end{array} \right ] x \leq \left [ \begin{array}{c} 0_{n-2} \\ 0_{n-2} \\ c  \end{array} \right ]&;& x \geq 0
\end{array}
\end{equation}

Where $0_n$ is the zero column vector of length $n$. Then we find the dual to be

$$
\begin{array}{ccccc}
{\displaystyle \mathop{\min}_{y \in \R^{2(n-2) +m}}  \left [ \begin{array}{c} 0_{n-2} \\ 0_{n-2} \\ c  \end{array} \right ]^T  y }& ; & \left [ \begin{array}{ccc} \Phi^{*T} & - \Phi^{*T} &  I_m
\end{array} \right ] y \geq \Phi_{1,j}^T &;& y \geq 0
\end{array}
$$

After further inspection this is equivalent to $n-2$ unrestricted in sign variables, one for each vertex $v(i) \neq s,t$ called $v_i$, and $m$ variables, one for each edge $e_j$ such that

\begin{equation}\label{dual}
\begin{array}{ccccc}
{\displaystyle \mathop{\min}_{[v,e] \in \R^{n-2+m}} c^T  e }& ; & \left [ \begin{array}{cc} \Phi^{*T}  &  I_m
\end{array} \right ] [v,e]^T \geq \Phi_{1,j}^T &;& e \geq 0
\end{array}
\end{equation}

These restrictions translate to the following set of inequalities: 

$$
v_i - v_j + e_k  \geq 0 : v(i) \neq s,t \,\, v(j) \neq s,t \,\, e(k)=(v(i),v(j)) : i<j
$$
$$
-v_i + e_k \geq 1 : v(i) \neq s,t \,\, e(k)=(s,v(i))
$$
$$
v_i + e_k \geq 0 : v(i) \neq s,t \,\, e(k)=(v(i),t)
$$
$$
e_k \geq 0 \,\, \forall k \in \{1, \ldots , m \}
$$

We can define $v_1= -1$ and $v_n = 0 $ and we write all the equations in the form

$$
v_i - v_j + e_k  \geq 0 : e(k)=(v(i),v(j)) : i<j
$$

\begin{lem} For any cut $(S, \bar S) $ there exists a feasible solution of (\ref{dual}) such that the value of the function at this feasible solution equals the capacity of the cut.
\begin{proof}
Let $(S,\bar S) $ be a cut. Define $e_k = 1$ if and only if $e(k)=(s,\bar s) $ with $s \in S$, $\bar s \in \bar S$ and $e_k =0$ in any other case. $v_i = -1 $ if and only if $v(i) \in S$ and $v_i=0$ in any other case. Then $c^T e = C(S,\bar S) $ and it is straightforward to check that the restrictions hold.\\
\end{proof} 
\end{lem}

As a corollary we get

\begin{cor} Lemma (\ref{maxleqmin})
\end{cor}

\begin{lem} For an optimal solution $[v^*,e^*]$ of (\ref{dual}) there exists a cut $(S,\bar S)$ such that $C(S, \bar S) \leq c^T e^*$.
\begin{proof} Let $\chi \in [-1,0] $ be a random variable with uniform distribution. Define a random variables    for each edge by

\begin{displaymath}
   e(k)_\chi = \left\{
     \begin{array}{rl}
       1 & : v_i^*\leq \chi \leq v_j^* \\
       0 & : \mbox{else} 
     
     \end{array}
   \right.
\end{displaymath} 

 Note that this assignment defines a random cut. If $v_i^* < v_j^* $ then $ \Prob (v_i^* \leq \chi < v_j^* ) = \min \{ 1, v_j^* - v_i^* \} \leq v_j^* - v_i^* \leq e_k^* $ then by the restrictions of the problem,  we get
$$
\mathbb{E} (C(S, \bar S )) = \sum c_k \mathbb{E}( e(k)_\chi) \leq \sum c_k e_k^* = c^T e^*
$$
As the expected value of the random cut capacity is less or equal to the optimal value of the problem, there exists a cut of capacity less or equal to the optimal value.\\
\end{proof}
\end{lem}

This proves that the dual of MAXFLOW is in fact a relaxation of MINCUT and we get, by strong duality

\begin{cor} Theorem (\ref{maxflowmincut})
\end{cor}

\newpage

\section{MAXFLOW algorithms}\label{algs}
\vspace{1cm}
\subsection{The Ford-Fulkerson algorithm}
L.R, Ford Jr. and D.R. Fulkerson devised a polynomial time algorithm to compute a maximal flow first published in 1962 \cite{Ford}. We introduce some new concepts needed to describe the algorithm, and prove some general facts about it.\\

\begin{defi} Given a network $(G,c)$ and a flow $f$ we define $\bar c, \bar f : V \times V \rightarrow \R $
\begin{displaymath}
   \bar c(u,v)= \left\{
     \begin{array}{rl}
       c(u,v) & : (u,v) \in E \\
       0  &: \mbox{else}
     \end{array}
   \right.
\end{displaymath} 

\begin{displaymath}
   \bar f(u,v)= \left\{
     \begin{array}{rl}
       f(u,v) & : (u,v) \in E \\
       -f(u,v)  &: (v,u) \in E\\
       0 &: \mbox{else}
     \end{array}
   \right.
\end{displaymath} 
\end{defi}

\begin{lem}\label{cons} For a given flow $f$ on a network $(G,c)$, for any fixed vertex $u \neq s,t $
$$
\sum_{v\in V}\bar f (u,v) = 0
$$

\end{lem}

\begin{rem}\label{remark} In fact, it is equivalent to define a flow as a function $f$ mapping $V\times V$ to the reals such that the equation in definition \ref{cons} holds, and such that $f(u,v)=-f(v,u)$ for every pair of vertices. From now on we refer to a flow in this sense, and we refer to $\bar f, \bar c $ as $f$ and $c$ whenever it does not cause confusion. Under this new definitions we have that the capacity of a cut can be written as 
$$
C(S,S')=\sum_{u \in S, v \in S'} c(u,v)
$$
and lemma (\ref{value off}) translates to
$$
|f|=\sum_{u \in S, v \in S'} f(u,v)
$$
for any cut $(S,S')$
\end{rem}

\begin{defi}\label{residual} For a network $(G,c)$ and a flow $f$ we define the \emph{residual capacity} of a pair of vertices $(u,v)$ as $c_f(u,v)=c(u,v)-f(u,v)$. Any such pair with residual capacity greater than zero is called a residual edge. Note that the residual capacity is always greater or equal to zero. We define the \emph{residual graph} $G_f$ as the graph with vertex set that of $V(G)$ and edge set the set of residual edges.

\end{defi}

\begin{defi}\label{aug} Given a flow $f$ on a network $(G,c)$, an \emph{augmenting path} is a directed path on $G_f$ from source to sink.
\end{defi}

\begin{lem} \label{nopath}
a flow $f$ is maximal if and only if there is no augmenting path on $G_f$.
\begin{proof}
Suppose there is no augmenting path on $G_f$. Let $S$ be the set of vertices $v$ such that there exists a directed path from $s$ to $v$ in $G_f$. Let $S'=V \backslash S $. $(S,S')$ is then a cut. By definition of S, we have that $c_f(u,v)=0 \Rightarrow c(u,v)=f(u,v)$ for any $u \in S$, $v \in S'$ then we have, following remark (\ref{remark})
$$
C(S,S') = \sum_{u \in S, \, v \in S'} c(u,v)=  \sum_{u \in S, \, v \in S'} f(u,v)= |f|
$$
So $f$ is a maximal flow by  Theorem (\ref{maxflowmincut}) or Theorem(\ref{weakduality}).\\

Now suppose there is an augmenting path $(v_0=s,v_1,\ldots, v_{k-1}, v_k=t)$. Let $A=\min \{ c_f(v_i, v_{i+1}): i=0,\ldots, k-1 \}>0$. Define $F:V\times V \rightarrow \R $ as $F(v_i,v_{i+1})=f(v_i,v_{i+1}) + A$, $F(v_{i+1},v_{i})=f(v_{i+1},v_{i}) - A$ and $F(u,v)=f(u,v)$ on any other pair of vertices. One can easily check that $F$ is a flow, and that $|F|=|f|+A$ so $f$ is not a maximal flow.\\
\end{proof}
\end{lem}

Now this theorem is the basic result needed to state the Ford-Fulkerson algorithm. Starting with the zero flow, as long as there exists an augmenting path with respect to such flow, we can increase the value of the flow by $A$ as defined in the above proof. 

\begin{lem} If $(G,c)$ is such that $c(u,v) \in \mathbb{Z} $ then the algorithm terminates.
\begin{proof} At each step of the algorithm, the value is increased by $A \geq 1 $ so a maximal flow is reached after a finite number of steps.\\

\end{proof}

\end{lem}

As corollaries we get

\begin{cor}\label{integer} If the capacities of a network are integers, then the value of the maximal flow is an integer and there exists a maximal flow with $f(u,v) \in \mathbb{Z} $ for every edge $(u,v)$.

\end{cor}

\begin{cor} If the capacities of a network are rational numbers, then the algorithm terminates.

\end{cor}

In fact there are examples of networks with irrational capacities such that the algorithm never terminates, moreover, the value of the flow in each step does not converge to the actual value of the maximal flow, so our algorithm must have as a condition that the capacity is at least a rational valued function. Then, the running time of the algorithm depends on the way the augmenting paths are chosen. There are many ways to find an augmenting path, like the shortest augmenting path or the largest bottleneck (value of $A$) augmenting path, that lead to a polynomial time algorithm.

\vspace{1cm}

\subsection{The Goldberg-Tarjan algorithm}\

The Goldberg-Tarjan algorithm \cite{Goldberg} is another polynomial time algorithm with a different approach to the problem of finding a maximal flow. Instead of increasing the flow along augmenting paths, it starts with a \emph{preflow}, which is a function on $V \times V$ which satisfies \emph{excess of flow} at each vertex, and then pushes \emph{excess flow} to edges closer to the sink. Next we formalize these concepts following Goldberg-Tarjan's article \cite{Goldberg}. \\

\begin{defi}\label{preflow} Given a network $(G,c)$ a preflow is a function $f: V\times V \rightarrow \R$ satisfying:
\begin{enumerate}[i)]
\item $f(u,v) \leq c(u,v)$
\item $f(u,v)=-f(v,u)$
\item for any vertex $x \neq s$, $ \sum_{v \in V} f(x,v) \geq 0$
\end{enumerate}  
\end{defi}

\begin{defi} \label{residual2} for a network $(G,c)$ and given a preflow $f$ on the network, we redefine the residual capacity of $(u,v) \in V \times V $ as $c_f(u,v)=c(u,v)-f(u,v)$. If $c_f(u,v) > 0 $ we call such pair a \emph{residual edge}. We define the residual graph as the directed graph having vertex set $V$ and edge set the set of residual edges.
\end{defi}

Note there are similarities with definition (\ref{residual}) but in this definition we are working with a preflow rather than a flow.

\begin{defi}\label{excess} the \emph{excess flow} at a vertex $x \in V$ is defined as $ \sum_{v \in V} f(x,v) \geq 0$.
\end{defi}

\begin{defi}\label{valid labeling} given a a \emph{valid labeling} on a network $(G,c)$ is a function $d: V \rightarrow \mathbb{Z}_{\geq0}\cup \{ \infty \} $ such that $d(s)=n$, $d(t)=0$ and $d(v) \leq d(w) +1 $ for every residual edge $(u,v)$.
\end{defi}

It can be shown that for any vertex $v$, if $d(v)<n$ then $d(v)$ is a lower bound on the distance from $v$ to $t$ in the residual graph and if $d(v) \geq n$ then $d(v)-n$ is a lower bound on the distance to $s$ in the residual graph \cite{Goldberg}. This labeling of the vertices permits the algorithm to push excess flow to vertices that are estimated to be closer to the sink and, if needed, to return flow to vertices estimated to be closer to the source.

\begin{defi}\label{active} a vertex $v$ is called active if $d(v) < \infty$ and $e(v) \geq 0 $. \end{defi}

Now we define the basic operations, push and relabel, that the main algorithm uses.

{\bf Push. } Let $(v,w)$ be such that $v$ is an active vertex, $c_f(u,v) > 0 $ and $d(v) = d(w)+1 $. Define $ \delta = \min(e(v), c_f(v,w)) > 0 $. Redefine $f(v,w):= f(v,w) + \delta $, $f(w,v):= f(w,v) - \delta $, $e(v):=e(v) - \delta$ and $e(w):= e(w)+\delta $.\\

{ \bf Relabel. }Let $v$ be an active vertex such that for any $ w \in V $, $c_f(v,w) > 0 \Rightarrow d(v) \leq d(w) $. Redefine $d(v):= \min \{ d(w)+1 \, : \, (v,w) \,  \mbox{residual edge} \} $.\\

As initial preflow we take the function $f$ such that for any $v \in V $, $f(s,v)=c(s,v) $ and zero everywhere else. It is readily checked that this is a preflow. As an initial labeling of the vertices we take $d(s)=n$ and zero everywhere else. As long as there is an active vertex $v$, either an operation of push or relabel is applicable to $v$. When there are no more active vertices the algorithm terminates, and the preflow becomes a flow, and in fact, it is maximal. Details of the proof of correctness and termination of the algorithm can be found in \cite{Goldberg}. We show only correctness assuming termination. 

\begin{lem} \label{reach} If $f$ is a preflow and $d$ is any valid labeling for $f$ then the sink $t$ is not reachable from $s$ in $G_f$.

\begin{proof} Suppose $(s=v_0, \ldots, v_k = t ) $ is a path from $s$ to $t $ in the residual graph. Clearly $k \leq n $. Now $ (v_i, v_{i+1} ) $ is a residual edge for every $i$. So by definition of valid labeling $d(v_i) \leq d(v_{i+1}+1) $ so we have $d(s) \leq d(v_1)+1 \leq \ldots \leq d(t) + k =k$ this contradicts the fact that $d(s)=n $. \\   
\end{proof}
\end{lem}

Now recall lemma (\ref{nopath}). 

\begin{teo} If the algorithm terminates and $d$ is a valid labeling for $f$ with finite labels, then $f$ is a maximal flow.  
\begin{proof}
At the end of the algorithm there are no active vertices, as the labels are finite, it means that all vertices have zero excess, so $f $ is a flow. By lemma (\ref{reach}) and lemma (\ref{nopath}) this flow is in fact maximal. \\
\end{proof}
\end{teo} 

One important remark about this algorithm is the fact that it always works (it terminates and it is correct) no matter what type of capacity function we are dealing with. The Ford-Fulkerson fails to terminate in some cases where the capacity function is not rational. It is also important to note that the algorithm relies only on \emph{local} operations, that means the operations depend and modify only parameters related to a small part of the graph, this allows a parallel implementation of the algorithm that takes advantage of multicore processors. A special implementation of such algorithm terminates after $O(n^2m)$ steps.

\vspace{1cm}

\subsection{Hochbaum's pseudoflow}

Dorit Hochbaum's \emph{pseudoflow algorithm} \cite{Hochbaum} is an algorithm with a different approach to the maximum flow problem. Instead of directly finding a maximum flow, it first solves the \emph{maximum blocking cut} problem, then a maximum flow is recovered. Although the most complicated of the three, it is also the most efficient. We follow \cite{Hochbaum}:

\begin{defi} A \emph{pseudoflow} $f$ on a given network $(G,c) $ is a function $f: V \times V \rightarrow \R $ such that 

\begin{enumerate}[i)]
\item $ f(v,w)=-f(w,v)$, $\forall (v,w) \in V\times V $
\item $0\leq f(v,w) \leq c(v,w) $, $\forall (v,w) \in E $
\end{enumerate}

\end{defi}

The concept of pseudoflow drops the conservation of flow constraint, preserves the capacity constraint on the edges of the graph and the antisymmetry constraint on $ V \times V $. We define the residual capacity and residual graph in the same manner we did with flows and preflows. \\

\begin{defi}

For a directed, weighted, simple graph $G=(V,E)$ with weights $w_v$ for each $v \in V$ and arc capacities $c(v,w) $ for each $(v,w) \in V \times V$,  we will define $G_{st}$ as a directed graph with vertex set $V_{st}=V \cup \{s, t\} $,edge set $E_{st}=E\cup A(s) \cup A(t) $ where $A(s)= \{ (s,v) : w_v > 0 \} $ and $A(t)= \{ (v,t) : w_v < 0 \} $ and arc capacities $c(s,v)=w_v$, $c(v,t)=-w_v$ and the other arc capacities left unchanged. Starting from $G_{st} = (V_{st}, E_{st}) $ we define the extended network $G^{ext}$ as the graph obtained from $G_{st}$ by identifying $s,t$ as a single vertex $r$ and adding the edges $ (r,v),(v,r) $ for every $v \in V_{st} \backslash \{ s,t\} $.

\end{defi}

 We define the excess of flow at a vertex $e(v)$ as in definition ( \ref{excess}). \\

Now we consider a pseudoflow $f$ on $G_{st}$  and a rooted spanning tree with root $r$, $T$ of $G^{ext}$ such that

\begin{enumerate}[i)]
\item $f$ saturates all arcs in $A(s) \cup A(t)$

\item For every arc in $E \backslash T $, $f$ is either zero or saturates the arc. 

\item In every branch all downward residual capacities are strictly positive.

\item the direct children of $r$ are the only vertices that do not have zero excess.

\end{enumerate}

\begin{defi}\label{normalized} a spanning rooted tree with root $r$ of $G^{ext}$ that satisfies the previous conditions is called a \emph{normalized tree}. Note that this is an undirected graph.
 
A child $r_i$ of $r$ is classified as: 
\begin{enumerate}[i)]

\item Strong if $e(v) >0 $

\item Weak if $e(v) \leq 0 $

\end{enumerate}

A vertex $v$ is called weak or strong if it has a weak or strong ancestor, respectively.

\end{defi}

As we mentioned earlier, Hochbaum's algorithm solves first the maximum blocking cut problem, which we state next:\\

{ \bf Problem: } For a directed, weighted graph $G=(V,E)$ with vertex weights $w(v)$ for each vertex $v$, and arc capacity function $c(a,b)$ defined for every $(a,b) \in E$, find $S\subset V$ such that
$$
surplus(S) = \sum_{v \in S} w(v) - \sum_{ \substack{a \in S \\ b \in \bar S}} c(a,b)
$$
is maximum. Such a set is called a maximum surplus set and $(S, \bar S)$ is called a maximum blocking cut.
 \\

The key is to find the relation between a maximum blocking cut in $G$ and a minimum cut in $G_{st}$. Given by the following lemma:

\begin{lem}\label{mincutmaxblock} $\{s\} \cup S$ is the source set of a minimum cut in $G_{st}$ if and only if $(S, \bar S ) $ is a maximum blocking cut in $G$.
\end{lem}

This is proven in \cite{Hochbaum} following an article by Radzik \cite{Radzik}. The following lemma, also found on the article \cite{Hochbaum}, is fundamental for the correctness of the algorithm.

\begin{lem}\label{optimal} For a normalized tree $T$ and pseudoflow $f$ on $G_{st}$ saturating $A(s)$ and $A(t)$ and a set of strong vertices $S$, if the residual capacity of any edge $(a,b)$ with $a \in S $ and $b \in \bar S $, $c_f(a,b)=0$ is zero then $S$ is a maximum surplus set and $(S,\bar S)$ is a maximum blocking cut.

\end{lem} For a normalized tree $T$ if the set of strong vertices $S$ satisfies the condition in lemma (\ref{optimal}) the tree is called \emph{optimal}\\

The algorithm starts with a normalized tree related to a pseudoflow $f$ on $G_{st}$. There are multiple choices of such a tree. We will start with a \emph{simple normalized tree}. It corresponds to a pseudoflow $f$ saturating $A(s)$ and $A(t)$ on $G_{st}$. In this normalized tree every vertex in $V$ forms an independent branch. The set of strong vertices are those adjacent to the source. \\

By lemma (\ref{optimal}), it is desirable to reduce the residual capacity from strong to weak vertices, therefore, with each iteration of the algorithm, a residual edge from $S$ to $\bar S$ is chosen, this is called a \emph{merger arc(edge)}. If such an edge does not exist then the tree is optimal and the set of strong vertices form a maximum blocking cut. If there is one, then such edge becomes a new edge of the tree and the edge joining the root of the strong branch to $r$ is removed from the tree. Then the excess of the root of the strong branch is pushed upwards until it reaches the root of the weak branch. Note that this path is unique.\\

It is not always possible to push the total of the excess along an edge.If there is an edge, say $(a,b)$ that does not have enough residual capacity to push the excess then such edge is removed (split) from the tree, $a$ (the tail of the edge) becomes the root of a new strong branch with excess equal to the excess pushed minus the residual capacity of the edge. This is done in such a way so that the property that only roots of branches may have nonzero excess is maintained through the running of the algorithm. The remaining excess at $b$ continues to be pushed in the same fashion until it reaches the root of the weak branch or until it reaches another edge that does not have enough residual capacity and the process is repeated. This process assures that the tree is normal at the end of each iteration. \\

Termination of the algorithm follows from the next lemma:\\

\begin{lem} At each iteration of the algorithm either the total excess of the strong vertices is strictly reduced or the number of weak vertices is reduced.

\begin{proof}
Recall that from the properties of definition (\ref{normalized}) we have that all downward residual capacities of edges are positive. After appending a merger edge to the tree and removing the edge joining the root of the strong branch $r_s$ to $r$, the path from $r_s$ to the weak branch becomes an \emph{upward} path with positive residual capacity at each edge of the path, then some positive amount of excess arrives at the weak branch that is being merged.  Then either some positive amount of excess arrives at the root of the weak branch and the total excess is strictly reduced, or there is some edge in the weak branch without enough residual capacity. In this case the edge is split and the tail of such edge becomes a strong vertex. Note that if some weak vertex becomes strong in this fashion, then all of its children, including the former strong branch, becomes strong. Then if such operation takes place, the number of weak vertices is strictly reduced.\\
\end{proof}
\end{lem}
Now let $M^+=C(\{s\}, V )$ be the sum of capacities in $A(s)$ and $M^-$ be the sum of capacities in $A(t)$ then by the final comment in the previous lemma we see that any iteration that reduces the total excess is separated from another iteration of such type by at most $n$ iterations. Then it follows immediately for integer capacities that

\begin{cor}
The complexity of the algorithm is $O(nM^+)$
\end{cor}

Now as the problem is symmetrical on $s$ and $t$ we find that by reversing all directions of the edges of the graph and interchanging $s$ and $t$ we get an equivalent problem so it follows again that for integer capacities

\begin{cor}
The complexity of the algorithm is $O(n * \min \{M^+, M^- \} ) $
\end{cor}

Correctness of the algorithm follows from lemma (\ref{optimal}) as at the end of the algorithm there are no merger arcs left.\\

Now in order to solve our initial problem we have to recover a maximum flow from the pseudoflow and maximum blocking cut obtained after the algorithm terminates, as it is not guaranteed that the pseudoflow becomes a flow after termination. In what follows we describe how to recover such maximum flow. 

\begin{defi} An $s-t$ path-flow on a network $(G,c)$ is a flow $f$ on $(G,c)$ such that the edges carrying a strictly positive amount of flow form an $s-t$ path on $(G,c)$. A cycle-flow on $(G,c)$ is a flow on $(G,c)$ such that the edges carrying a strictly positive amount of flow form a directed cycle on $(G,c)$.
\end{defi}

\begin{teo}\label{flowdecomp} {\bf (Flow decomposition)} Let $f$ be a flow on $(G,c)$, then $f$ can be decomposed as the sum of at most $m$ $s-t$ path-flows and cycle-flows. 
\begin{proof}
Suppose $f$ is such that $|f| > 0$ then there is some $(s,v_1) \in E$ such that $f(s,v_1)>0$. If $v_1=t$ we have a directed $s-t$ path and we define a flow $f_0$ carrying an amount of flow $f(s,v_1)$ on such a path and zero everywhere else. If $v_1 \neq t$ then there exists some edge $(v_1,v_2)$ with some positive amount of flow as a result of conservation of flow. In this way we construct an $s-t$ path (we may suppose it has no loops) and we define the flow $f_0$ as carrying an amount of flow equal to the minimum of the flow over the edges of this path and zero everywhere else, it is readily checked that this is a feasible flow. $f'=f-f_0$ is again a feasible flow where $f(u,v)>0$ and $f'(u,v)=0$. Using the same argument for $f'$ we arrive at a flow $f^*=f-f_0-\ldots f_k$ with zero net flow. If $f^*$ is not the zero flow, then analogously to the previous argument we may construct a cycle on the graph and define a cycle-flow as the minimum over the flow of the edges on the cycle and zero everywhere else, this is a feasible flow $f_{k+1}$ and $f^*-f_{k+1} $ has some new edge with zero flow. We continue in such fashion and arrive at $f-f_0- \ldots f_h=0$ so $f=f_0+\ldots f_h$ where $f_i$ is either a path-flow or a cycle-flow. $h\leq m$ as at least the flow on one edge becomes zero in each step.\\ 
\end{proof}

\end{teo}

\begin{lem}\label{strong path}(see \cite{Hochbaum}) For any strictly strong node there exists a residual path either to the source or to some strictly weak node.

\end{lem}

In order to use the flow decomposition theorem first we have to consider a network $(G',c')$ related to $(G,c)$ such that the preflow $f$ becomes a feasible flow. This is done by considering a super source $\bar s$ and supersink $\bar t$, adding edges $(\bar s, s) \cup \{ (\bar s,v): \mbox{ v is strictly weak} \} $ with flow and capacity equal to the deficits on such vertices, and edges $(t, \bar t) \cup  \{ (v, \bar t): \mbox{ v is strong} \} $ with capacity and flow equal to the excesses on such vertices. The flow on any other edge has the same value as the preflow. This function is now a feasible flow on the network with source $\bar s$ and sink $ \bar t$. \\

To get a feasible flow on the original network, we have to get rid of excesses at strong nodes and deficits at strictly weak nodes. For any strong vertex $v_s$, as long as $f(v_s, \bar t) >0 $ we have that $(\bar t, v_s) \in E_f$ is part of the residual network. Hence by lemma (\ref{strong path}) we have a residual path from $\bar t$ to $\bar s$ that contains the edge $(\bar t, v_s)$. Increasing the flow on such path by an amount of $\delta$ equal to the minimum over the residual capacities of the path, actually decreases the excess of $v_s$ by the same amount. After one such step, either the vertex $v_s$ arrives at zero excess or this process can be repeated by lemma (\ref{strong path}). This is a process analogous to flow decomposition on the reversed graph. After termination there are no vertices other than $t$ with positive excess.\\

In the same fashion, the remaining flow is decomposed until positive deficits at strictly weak vertices are disposed. This must be done via $t$ as it is the only vertex sending a positive amount of flow to $\bar t$. After termination all vertices except $s$ and $t$ have nonzero excess. Deleting $\bar s, \bar t$ from the graph leaves us with a feasible flow on $G_{st}$.

\begin{cor} A maximum flow can be recovered from an optimal normalized tree with pseudoflow f.
\begin{proof}
For an optimal tree we have $C_f(S, \bar S)=0$ that is, there are no residual edges directed from strong to weak vertices. Hence, following the previous argument, excesses at strong vertices can be disposed using only paths traversing strong nodes. Now there are no edges directed from a weak to a strong vertex with positive flow, as otherwise the reverse edge would have residual capacity greater than zero, a contradiction. So by the proof of theorem (\ref{flowdecomp}) the remaining deficits at weak vertices are disposed using only paths traversing weak vertices. It then follows that $C_f(S, \bar S)=0$ after recovering a flow $f$ so that $c(v,w)=f(v,w) $ for $v \in S, w \in \bar S$ and as consequence $|f|= C(S, \bar S) $. By lemma (\ref{mincutmaxblock}) and lemma (\ref{optimal}) $(S, \bar S) $ is a minimum cut. This shows $|f|$ is maximum.
 
\end{proof}

\end{cor}

\newpage

\section{Applications}

There are many not so obvious applications of maximum flow algorithms and results to different pure and applied topics, we show three interesting problems that can be solved using the previous results.

\subsection{Hall's Marriage Theorem}

Let $G=( V\cup W,E)$ be a bipartite graph, where $V \cap W = \emptyset$ and $|V|=|W|=n$. Label the vertices in $V$ as $v_1, \ldots , v_n $, and the vertices in $W$ as $w_1, \ldots, w_n $. A perfect matching on G is a permutation $\sigma \in S_n $ such that $[v_i,w_{\sigma(i)}] \in E$ for every $i=1, \ldots, n $.\\

\begin{defi} Let $S \subset V$.  $N(S):=\{ w \in W : [v,w] \in E \mbox{ for some } v \in S  \} $ is the set of neighbors of S.

\end{defi}

 We prove the following using the Maxflow-Mincut theorem (\ref{maxflowmincut}).\\

\begin{teo}\label{Halls} {\bf (Hall's Marriage Theorem)}
A perfect matching exists if and only if
$$
\forall S\subset V, \, |S| \leq |N(S)|
$$

\begin{proof}
Clearly such condition is necessary as $\sigma$ is injective. Suppose $\forall S\subset V, \, |S| \leq |N(S)|$. Now we construct an $s-t$ network by directing all edges $e\in E$ from $V$ to $W$, adding a source $s$ and sink $t$ and appending the edges $\{ (s,v) : v \in V \} \cup \{ (w,t) : w \in W \}$. We set the capacity of such new edges to 1, and the capacity of the original edges to $n+1$. Let $S$ be a minimum cut on such network. We show that $C(S, \bar S)=n$. $C(S,\bar S) \leq n$ as the cut $S'=\{s\}$ has capacity $n$ and $C(S,\bar S)$ is minimum. Now we show $C(S, \bar S) \geq n$. Let $X=S \cap V $. $N(X) \subset W $ and if $N(X) \not \subset S$ then there would be an edge crossing the cut, of capacity $n+1$ so $C(S, \bar S) \geq n+1$ then $N(X)\subset S \cap W$. On the other hand, all edges traversing the cut are of the form $(s,v')$ where $v' \in V \backslash X$ or of the form $(w',t)$ where $w' \in S \cap W  $. Then
$$
C(S,\bar S)= \sum_{v' \in V \backslash S}c(s,v') + \sum_{w' \in S \cap W}c(w',t)= \underbrace{|V\backslash S|}_{=n-|X|} + \underbrace{| S \cap W|}_{\geq N(X)} \geq n-|X|+\underbrace{N(X) }_{\geq |X|} \geq n-|X|+|X|
$$
So the capacity of a minimum cut is $n$. By theorem (\ref{maxflowmincut}) and (\ref{integer}) there exists a maximum flow $f$ of integer values and net flow $|f|=n$. As there are only $n$ edges out of the source $s$ and into the sink $t$, and they have capacity $1$, they must be saturated. By conservation of flow and the fact that the flow is integer, for any $v \in V$ there exists only one $w\in W$ such that $f(v,w)=1$. Again by conservation of flow and integrality, for any $w \in W$ there exists only one $v \in V$ such that $f(v,w)=1$. This shows that the edges directed from $V$ to $W$ carrying a flow of $1$ define a perfect matching on $G$.\\
\end{proof}
\end{teo} 

\begin{cor} There exists a polynomial time algorithm that finds a perfect matching on a bipartite graph.

\end{cor}

\subsection{Counting disjoint chains in finite posets}
\begin{defi}
A finite poset $P:=(P', \leq)$ is a finite set $P'$ together with a partial order $ \leq $ on $P'$. We say that $P$ has $\hat 0$ or ($\hat 1$) if there exists an element $x \in P$ such that $x \leq y$ or $x \geq y$ for any $ y \in P$, respectively. A chain is a subset $c:=\{x_0, \ldots, x_n\} \subset P $ such that for any two elements $x_1,x_2$ either $x_1 \leq x_2 $ or $x_2 \leq x_1 $. \end{defi}

Given a finite poset $P$, we say that a chain $C$ is maximal if $C\cup \{x \} $ is not a chain for any $ x \in P \backslash C $. Clearly any maximal chain contains $\hat 0$ and $\hat 1$. We say that $y$ covers $x$ in the poset if $x<y$ and there exists no $z$ such that $x<z<y$. We say that a set $\{ C_i \}$ of chains are cover-disjoint if whenever $y$ covers $x$ then $ \{x, y\} $ belongs to at most one chain $C_i$. We would like to find a subset $S$ of the set of maximal chains, such that $S$ is cover-disjoint and such that $|S|$ is maximum.  \\

One of the possible ways of doing this is to work in a \emph{greedy algorithm} fashion, finding one of such chains and then repeating the process in the remaining part of the poset. We note that this may not lead to a partition of maximum size, as the example in Figure 2 suggests.

\begin{figure}[h!]
\begin{center}
\includegraphics[scale=0.5]{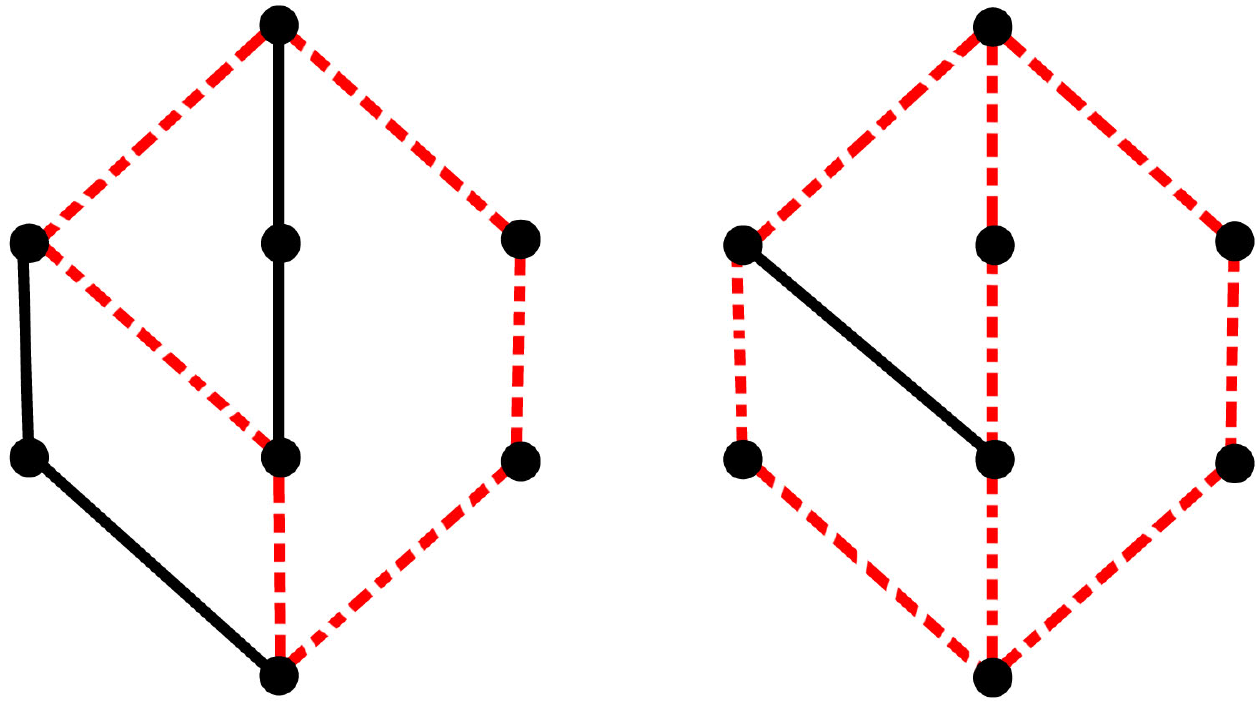} 
\caption{Two sets of maximal chains in a poset}
\end{center}
\end{figure} 

Instead, we consider an associated network $P_{st}$ where $s=\hat 0 $, $t = \hat 1$, $V=P$ and $(x,y) \in E $ whenever  $y$ covers $x$ . We define a capacity function with value $1$ on every edge.

\begin{lem}
The maximum number of disjoint chains in $P$ is equal to the net flow of a maximum flow in $P_{st}$.
\begin{proof}
We first show that given a set of disjoint chains $C_1, \ldots, C_k$ we can find an associated feasible flow on $P_{st}$ with flow value equal to $k$. $C'_{i} \{\hat 0, \hat 1 \} \cup C_i  $ forms a chain from $ \hat 0 $ to $\hat 1 $ and a directed path from source to sink in $P_{st}$. Define a flow $f_i$ in $P_{st}$ as having value 1 on the edges of $C'_i$ and zero everywhere else. This is a feasible flow. $f=f_i+\ldots + f_k$ is a function that satisfies conservation of flow and, as the chains were disjoint, it also satisfies capacity constraints so it is a feasible flow. Each $f_i$ saturates one edge leaving the source, hence $|f|=k$. This shows that $k$ is always less than the value of a maximum flow.\\

Now given a maximum flow $f$ on $P_{st}$ we construct a set of disjoint chains of size $|f|$. We may assume $f$ has integer values by corollary (\ref{integer}). By theorem (\ref{flowdecomp}) we may write $f= f_1 \ldots f_j$ where each $f_i$ is an $s-t$ path. As the flow has integer values, so do $f_1, \ldots f_j$. As the capacities are all equal to $1$ each $f_i$ must have net flow equal to one and so the $f_i$'s do not intersect as they saturate all the path. Then $\{f_1, \ldots, f_j \} $ define a set of disjoint $\hat0,\hat1$ chains. Finally $|f|=|f_1|+\ldots+|f_j|=j$.\\

\end{proof}
\end{lem}

\subsection{Image segmentation}
The problem of segmenting a given image is that of defining a partition of the pixels as two sets, the foreground and the background, so that they form coherent regions. There are multiple other problems defined under the label of image segmentation. In the following we show how to define the problem and how to solve it using algorithms of flow optimization in a network, following T.M. Murali's lecture notes \cite{Murali}. Throughout this section we denote a directed edge as $(v,w)$ and an undirected edge as $[v,w]$.\\

We define a finite undirected graph $G=(V,E)$ where $V$ is the set of pixels of an image, $V \subset \Z_+ \times \Z_+ $ and the set of edges $E$ comprises the set of neighbors for each pixel $(x,y) \in V$. The set of neighbors of $(x,y)$ is $N_{(x,y)} :=  \{ (x+1,y), (x-1,y), (x,y+1), (x,y-1) \} \cap V $. We define functions $a: V \rightarrow [0,1] $, $v \mapsto a_v= \mbox{ probability that }v \mbox{ is in the foreground} $, $b: V \rightarrow [0,1] $, $v \mapsto b_v= \mbox{ probability that }v \mbox{ is in the background} $ and a penalty function $ p: E \rightarrow \R_+ $, $[v,w] \mapsto p{[v,w]}= $ penalty for defining $v,w$ for defining $v$ in the foreground and $w$ in the background.

{\bf Problem: } Partition the set $V$ as two sets $A,B$ (foreground/background) such that the function
$$
s(A,B):= \sum_{v \in A} a_v + \sum_{w \in B}b_w - \sum_{\substack{[v,w] \in E \\ |A \cap [v,w]|=1}} p_{[v,w]}
$$

is maximized.The idea is that if $a_v > b_v$ it's preferable to set $v$ as in the foreground and if a pixel $v$ has most of it's neighbors defined as in the foreground, it is preferable to set $v$ as in the foreground also. Such probabilities are given in the problem, however, different choices of such values may lead to better or worse results in the segmentation of the image. For instance if one is interested in isolating a small object in a big background, the best choice is to take higher values for the probability function $a_v$. \\

In order to construct such sets, one must define as foreground(background) vertices those with higher probability of belonging to the foreground(background), while reducing the total penalty of the boundary between foreground and background. We want to formulate this problem as a Mincut problem. In order to do this we have to overcome some difficulties, namely, that of working with an undirected graph rather than a capacitated network, and a function to be maximized rather than minimized.

\begin{lem}\label{maxmin} Let $Q=\sum_{v\in V}a_v +b_v$ then 
$$
s(A,B)=Q- \sum_{v \in A} b_v - \sum_{w \in B}a_w - \sum_{\substack{[v,w] \in E \\ |A \cap [v,w]|=1}} p_{[v,w]}=Q-s'(A,B)
$$
where
$$
s'(A,B)=\sum_{v \in A} b_v +  \sum_{w \in B}a_w + \sum_{\substack{[v,w] \in E \\ |A \cap [v,w]|=1}} p_{[v,w]}
$$
Then maximizing $s(A,B)$ is the same as minimizing $s'(A,B)$

\end{lem}

Now we consider a directed graph $G^*=(V \cup \{s,t\}, E^* \cup A(s) \cup A(t) \}$ where $E^*=\{ (v,w) : [v,w] \in E \}$, $A(s)=\{(s,v) : v \in V\}$ and $A(t)=\{ (v,t) : v \in V \}$ and we define a capacity function as $c(s,v)=a_v$, $c(v,t)=b_v$ and $c(v,w)=p[v,w]$. We then have a network where the source(sink) is connected to each pixel with such edge with capacity equal to $a_v$ ($b_v$) and where each undirected edge $[v,w]$ of neighbor pixels is replaced with two \emph{antiparallel} edges $(v,w),(w,v)$ both with capacity equal to $p[v,w]$. Then it follows immediately that for a cut $(A, B)$ in such network we have

$$
C(A,B)=\sum_{v \in A} b_v +  \sum_{w \in B}a_w + \sum_{\substack{[v,w] \in E \\ |A \cap [v,w]|=1}} p_{[v,w]}=s'(A,B)
$$

By lemma (\ref{maxmin}) arrive at next result:

\begin{cor} A minimum cut $(A,B)$ in $G^*$ solves the problem of image segmentation.
\end{cor}

There's only one difficulty left to overcome, as we must deal only with simple graphs, we must replace the set of antiparallel edges. We do this by adding for each pair of neighbor vertices $v,w$ two new vertices $c_{vw}, c_{wv} $ and replacing the antiparallel edges with the edges $(v,c_{vw}), (c_{vw}, w), (w, c_{wv}), (c_{wv}, v) $ all with capacity equal to $p[v,w] $. It use readily checked that it is equivalent to find a maximum flow on this new graph. \\

\begin{figure}[h!]
\begin{center}
\includegraphics[scale=0.5]{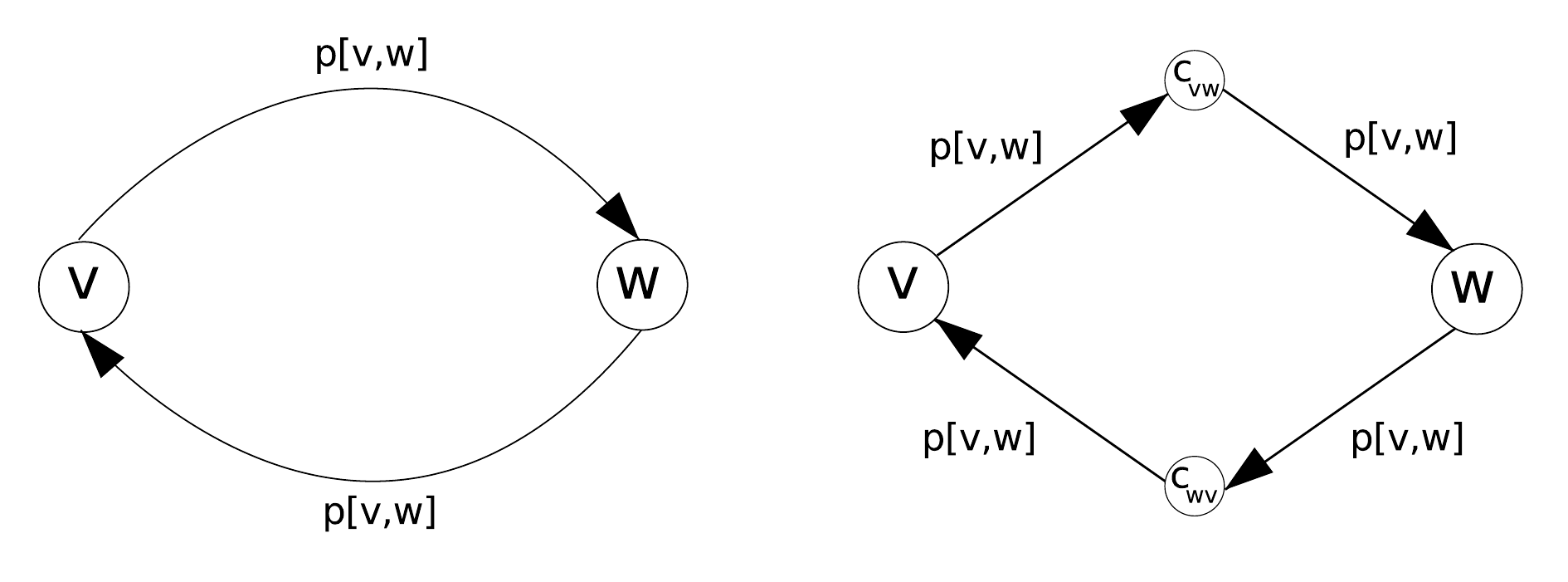} 
\caption{How to get a simple graph}
\end{center}
\end{figure}

Then we solve the problem by finding a maximum flow on such network using either the Ford-Fulkerson or Goldberg-Tarjan algorithm and then recovering a minimum cut using theorem (\ref{nopath}).

\newpage

\section{A Generalization of Maxflow}

We would like to define a more general optimization problem that reduces to the Maxflow problem on graphs and then try to generalize the optimization algorithms studied on previous sections. 

\subsection{Preliminaries}

\begin{defi} A simplex on a set $S$ is a finite subset $S' \subset S$.
\end{defi}

\begin{defi} A simplicial complex $\Delta$ on a set $S$ is a set of simplices on $S$ closed under taking subsets. Elements of $\Delta$ are called faces. Maximal faces (faces that are not subsets of any other face) of the complex are called facets. Elements of a simplex $X$ are called its vertices. The dimension of a face $X$ is defined as $\dim(X)=|X|-1$. The dimension of the complex $\dim(\Delta)$ is defined as the maximum over the dimension of its faces.

\end{defi}

\begin{defi} A simplicial complex $\Delta$ is called \emph{pure} if all facets have the same dimension.

\end{defi}

Given a network $(G,c)$ we can consider the graph $G^*$ that results from appending the edge $(s,t)$ with infinite capacity and then, the problem of finding a maximum flow on the original network is equivalent to finding a maximum \emph{circulation} on $G^*$, that is a positive function on the edges of $G^*$ satisfying capacity constraints and flow constraints on every vertex. In this case the objective function $|f|$ is the amount flowing through the edge with infinite capacity.\\

\begin{defi}Let $X=\{v_0,\ldots, v_d \}$ be a simplex with $d \geq 1$. Consider the set of \emph{orderings} of vertices of $X$,  $O(X):=\{ [v_{\sigma(0)}, \dots, v_{\sigma(d)}] : \sigma \in S_{d+1} \} $ modulo the relation $[v_0, \ldots, v_d]=[v_{\sigma(0)}, \dots, v_{\sigma_(d)}] \Leftrightarrow \sigma \mbox{ is even} $. This partitions the set in two equivalence classes that we call \emph{orientations} of $X$. To choose an orientation for $X$ is to choose one of such orientations, which we call the positive orientation and we say that $X$ is oriented. We denote an oriented simplex as $X=(v_0, \ldots, v_d)$.

\end{defi}

\noindent {\bf Notation.} For a simplicial complex $X$ we let $X^{(d)}$ be the set of its $d$-dimensional simplices. 

\begin{defi} For $d \geq 1$ Let $C_d$ be the free abelian group over the orderings of elements in $X^{(d)}$ modulo the relations $[v_0,\ldots,v_d]=[v_{\sigma(0)}, \dots, v_{\sigma(d)}] \Leftrightarrow \sigma \mbox{ is even} $ and $ [v_0,\ldots,v_d]=-[v_{\sigma(0)}, \dots, v_{\sigma(d)}] \Leftrightarrow \sigma \mbox{ is odd}$.  $C_0$ is defined in the same fashion but notice that the relations become trivial.

\end{defi}

\begin{defi} The boundary operator $\partial_d: C_d \rightarrow C_{d-1} $ is a homomorphism defined in the basis as
$$
\partial_d (X) =\sum_{i=0}^d(-1)^i [v_0,\dots ,v_{i-1},\hat{v_i},v_{i+1},\dots ,v_d]
$$
where $\hat{v_i}$ means deleting such term.
\end{defi}

Elements of $C_d$ are called $d$-chains. Elements of the subgroup $\ker ( \partial_d)$ are called $d$-cycles.

\subsection{Higher Maxflow}

\begin{defi} A $d$-dimensional network is a triple $(X,T,c)$ where
\begin{enumerate}
\item{$X$ is a simplicial complex of pure dimension $d$ all of whose facets have chosen orientations.}
\item{$T\in X^{(d)}$ is a distinguished oriented simplex of dimension $d$ satisfying the source condition:
\begin{itemize}
\item{For every oriented $d$-simplex $\sigma$ which intersects $T$ in a $(d-1)$-dimensional simplex $\tau$ the signs of $\tau$ in $\partial \sigma$ and $\partial T$ are opposite.}
\end{itemize}
}
\item{$c:X^{(d)}\rightarrow \mathbb{R}_+$ is a function with $c(T)=\infty$.} 
\end{enumerate}
\end{defi}

 \begin{rem} The source condition is a generalization to the assumption that on a network $(G,c)$ every edge incident with the source is directed out of it and every edge incident with the sink is directed into it. The source simplex $T$ is the generalization of an appended edge directed from sink to source with infinite capacity.\end{rem}

\begin{defi} A flow on a network $(X,T,c)$ is a function $f: X^{(d)}\rightarrow \mathbb{R}_+$ satisfying the following properties:
\begin{enumerate}
\item{ $f$ is a weighted cycle, that is
\[\sum_{\sigma\in X^{(d)}}f(\sigma)\partial{\sigma}=0\]
}
\item{ For every $d$-dimensional simplex $\sigma$ we have $0\leq f(\sigma)\leq c({\sigma})$}
\end{enumerate}
\end{defi}

\begin{rem} The condition that $f$ is a weighted cycle is a generalization of the conservation of flow condition. To see this we define the following: \end{rem}

\begin{defi}
Let $v=\{v_0, \ldots, v_{d-1} \} \in X^{(d-1)} $ be a $d-1$ dimensional simplex of a simplicial complex $X$ of degree $d$. Fix an orientation $[v]=[v_0, \ldots, v_d]$ of v. Let $e \in X^{(d)}$
\begin{displaymath}
   \phi (v,e) = \left\{
     \begin{array}{rl}
       1 & :  +[v] \mbox{ appears in } \partial e \\
       -1 & : -[v] \mbox{ appears in } \partial e \\
       0 &:  \mbox{else}
     \end{array}
   \right.
\end{displaymath} 

\end{defi}

This is, in fact, a generalization of the incidence function defined on section (\ref{flowian}). Then the condition that $f$ is a weighted cycle is equivalent to:\\

\noindent \emph{ for any $d-1$ dimensional oriented simplex $v$ }

$$
\sum_{e \in X^{(d)} }\phi(v,e)f(e) =0
$$

\begin{defi} The amount carried by a flow $f$ in a $(d)$-dimensional network $(X,T,c)$ is the number $f(T)$.
\end{defi}

This comes from the fact that after appending the edge $(t,s)$ to a network $(G,c)$ with flow $f$ one must define $f(t,s)=|f|$ so that conservation of flow holds in all vertices including $s,t$.

{\bf (HMax-Flow.)} The higher max flow problem asks to  find the maximum possible amount $f(T)$ which can be carried by a flow on a network $(X,T,c)$.

\begin{rem} A $(1)$-dimensional network is a capacitated graph ($T$ is the edge from $t$ to $s$) and HMax-Flow reduces to Max-Flow on graphs. 
\end{rem}

\begin{rem} If all the capacities are $1$ and $X$ is a triangulated orientable $d$-manifold and $T$ is any top-dimensional simplex then every top-dimensional cycle is a flow with $f(T)=1$. This is one HMax-flow. This follows directly from the definition of oriented manifold.
\end{rem}

\subsection{As an LP problem}
Even in higher dimension, the problem can be stated as a set of linear equalities and inequalities in a finite dimensional vector space, so it can also be stated as a linear program as we did in section (\ref{lp}). We continue with the convention that $|n|=|X^{(d-1)}|$ is the number of \emph{$(d-1) faces$}, and $m=|X^{(d)}|-1$ is the number of \emph{facets} without considering the source $T$. Then there are $2m$ flow restrictions, two for each edge, and $n$ flow conservation restrictions, one for each vertex. As always we consider fixed enumerations of the $(d-1)$ faces $v_i$ and of the facets $e_i$ (suppose $v_{m+1} = T $). After fixing an orientation of the $(d-1)$ faces, we define the matrix $[\partial]_{ij}=\phi(v_i, e_j)$, it has dimension $n\times m+1$. Let $I_m$ be the identity matrix of dimension $m$ and $I_m, 0$ be such matrix with a zero column vector appended as the rightmost column. Then we can state the problem as :

\begin{equation}\label{primal}
\begin{array}{ccccc}
{\displaystyle \mathop{\max}_{x \in \R^{m+1}}  x_{m+1} }& ; & \left [ \begin{array}{c} \partial \\ - \partial \\ I_m,0
\end{array} \right ] x \leq \left [ \begin{array}{c} 0_{n} \\ 0_{n} \\ c  \end{array} \right ]&;& x \geq 0
\end{array}
\end{equation}

where $c \in \R^{m} $ is the vector $c_i=c(e_i) $.\\

Computing the dual we find it can be stated as
$$
\min_{(v,e) \in \R^{n+m}} c^Te, \, \, \, \left[ \partial ^T \begin{array}{c} I_m \\ 0 \end{array} \right] (v,e) \geq e_{m+1}, \, \, e \geq 0
$$
where $v \in \R^n$ and $e \in R^m$.

\subsection{Further examples and conjectures}

\noindent {\bf Example: } Consider the three dimensional simplex. It is a simplicial complex on the set $\{1,2,3,4\}$ where the elements of the complex are the faces of the simplex.

\begin{figure}[h!]
\begin{center}
\includegraphics[scale=0.7]{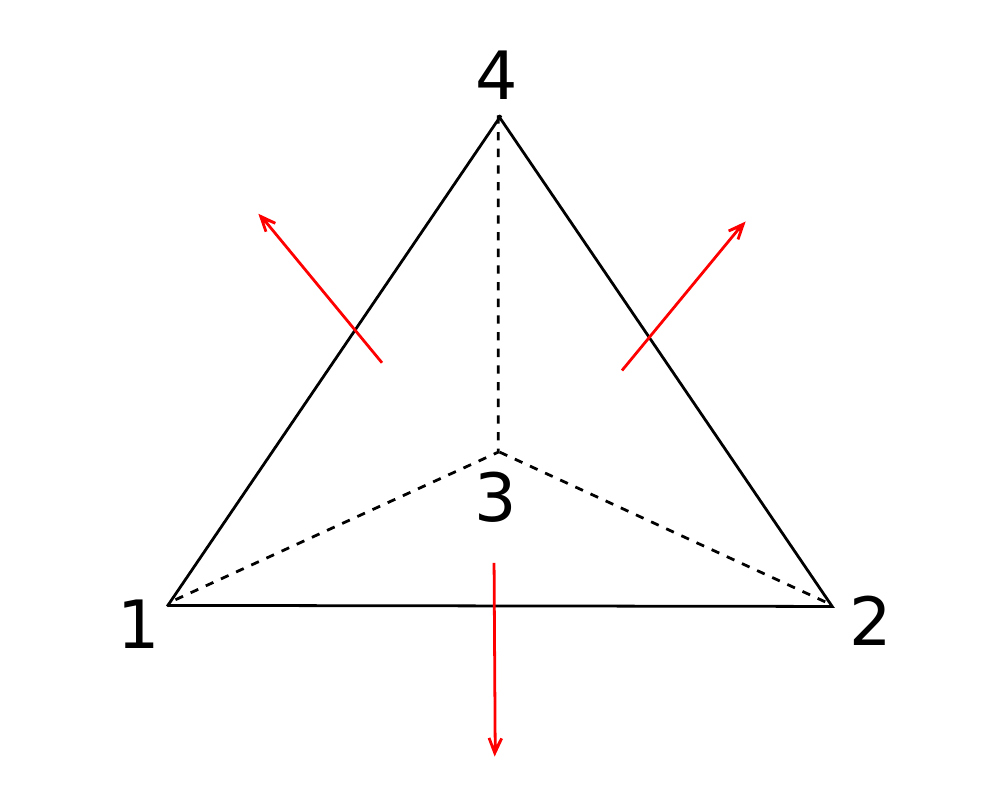} 
\caption{Oriented simplicial complex, the tetrahedron $Sim_2$}
\end{center}
\end{figure} 

Take the orientation as the induced by the outward normal. With this orientation the oriented faces are $ [234], [124], [143], [132] $. We choose the orientations of the one dimensional faces to be $[12], [14],[24], [13], [34], [23] $. The capacity on each facet is chosen to be 1. In this example we drop the commas, dealing only with integers less than 10. It is readily checked that $[132]$ satisfies the source condition so we may take $T=[132]$. With the enumeration of facets and $(d-1)$ faces given by this order we have that the incidence matrix is given by:

 \begin{table}[h!]
  \begin{center}
 $$ \left [  \begin{tabular}{ c c c c c }
   
    0 & 1 & 0 &-1  \\
    0& -1 & 1 & 0 \\
    -1 & 1 & 0 & 0  \\
    0 & 0 & -1 & 1 \\
    1 & 0 & -1 & 0 \\
    1 & 0 & 0 & -1
  
    \end{tabular} \right ] $$
  \end{center}
  \caption{Incidence matrix of $Sim_2$}
\end{table}

Towards a good definition of a cut on a generalized $d$-dimensional network we find feasible solutions of the dual of such problem. We are mostly interested in integer solutions to the problem. In this example we find both $\partial$ and $\left[ \partial^T \begin{array}{c} I_m \\ 0 \end{array} \right] $ to be totally unimodular matrix so the existence of integer optimal solutions is assured. We suggest the following definition. \\

\begin{defi} A cut on a generalized $d$-dimensional network $X$ is a partition $(S,S')$ of the $d-1$ faces of $X$.

\end{defi}

From a cut we can construct a solution to the dual program of HMaxflow. Let $(S,S')$ be a cut. Assign to the first $n$ dual variables $\lambda_v$ (where each one corresponds to a  $d-1$ face $v$ of the complex) the value $0$ if $v \in S$ and the value $1$ if $v \in S'$.  For each facet $\sigma$ of the complex, define $+(\sigma)$ be the set of $d-1$ faces such that $\tau$ induces their positive orientation and define $-(\sigma)$ in the same fashion. In the dual program, the variables that correspond to the $d-1$ faces are unrestricted in sign, while the variables that correspond to the facets $\eta_\sigma$ must be nonnegative. There is one inequality in the dual program for each facet that can be written in the form
$$
\sum_{v \in +(\sigma)} \lambda_v -\sum_{v \in -(\sigma)} \lambda_v + \eta_\sigma \geq 0
$$

for any $\sigma \neq T$, and
$$
\sum_{v \in +(T)} \lambda_v -\sum_{v \in -(T)} \lambda_v + \eta_T \geq 1
$$

So, given the cut and the assigned values to the $d-1$ faces variables we can define the dual variables corresponding to the facets $\eta_\sigma$ as

\begin{displaymath}
   \eta_\sigma = \left\{
     \begin{array}{rl}
       0 & : \sum_{v \in +(\sigma)} \lambda_v -\sum_{v \in -(\sigma)} \lambda_v \geq 0  \\
       -(\sum_{v \in +(\sigma)} \lambda_v -\sum_{v \in -(\sigma)} \lambda_v)  & : \mbox{else}
           \end{array}
   \right.
\end{displaymath} 
for $\sigma \neq T$ and
\begin{displaymath}
   \eta_T = \left\{
     \begin{array}{rl}
       0 & : \sum_{v \in +(\sigma)} \lambda_v -\sum_{v \in -(\sigma)} \lambda_v -1 \geq 0   \\
       -(\sum_{v \in +(\sigma)} \lambda_v -\sum_{v \in -(\sigma)} \lambda_v -1)  & : \mbox{else}
           \end{array}
   \right.
\end{displaymath} 

In such way it is readily checked the solution is dual feasible. We define the capacity of a cut $(S,S')$ as the value of the dual objective function at this solution, which is the weighted sum of the capacities of facets $\sigma \neq T $, with weights $\eta_\sigma$. In such way the capacity of a cut is always an upper bound to the value of a maximum flow, by weak duality. It is natural to ask whether the minimum of the capacities over all cuts equals the value of a maximum flow. We attempt to show this using an analogous probability argument as in section (2).

As in section (\ref{algs}) we can extend the capacity function defined on $X^{(d)}$ to the set of all orientations of elements of $X^{(d)}$ as $\bar c (x)=c(x) $ if $x\in X^{(d)}$ and $\bar c(x)=0 $ if $-x \in X^{(d)}$ and $\bar f (x)=f(x) $ if $x\in X^{(d)}$ and $\bar f(x)=-f(-x) $ if $-x \in X^{(d)}$.

\begin{defi} 
The residual complex $X_f$ is the (multi) simplicial complex whose facets are those $x \in X^{(d)}\cup -X^{(d)}$ such that the residual capacity of x $c_f(x):=\bar c(x)-\bar f (x) >0$. By definition, if the capacity function is not identically zero, the residual complex is a pure multisimplicial complex of dimension $d$.
\end{defi}

Recall from that from lemma (\ref{nopath}), a flow is not maximal if there exists a simple (no loops) path from $s$ to $t$ in the residual graph. After appending the edge $(t,s)$ this reduces to: a flow is not maximal if there is no simple cycle containing $(t,s)$ on the residual graph. This motivates the following definition:

\begin{defi}\label{augmentingcycle} 
Given a $(X,T,c)$ a $d$ dimensional network and a flow $f$ on $X$, an augmenting cycle is $d$-cycle $\sigma$ such that  $\sigma= \sum_i c_i X_i $ with $c_i \in \Z_+$ and $X_i \in X_f $ and such that $X_i=T$ for some $i$. \end{defi}

\begin{lem}\label{sum} Let $(X,T,c) $ be a network and $f, f'$ be two feasible flows in such network. Then $f+f'$ is positive and it is  a weighted cycle.
\begin{proof}

$$\sum_{\sigma\in X^{(d)}}(f+f')(\sigma)\partial{\sigma}= \sum_{\sigma\in X^{(d)}}f(\sigma) \partial{\sigma} +  \sum_{\sigma\in X^{(d)}}f'(\sigma) \partial{\sigma}=0 $$
The fact that $f+f'$ is positive is clear.
\end{proof}

\end{lem}

\begin{lem}\label{notmax} A flow is not maximal if there exists an augmenting cycle.
\begin{proof}
Let $A= \sum c_i X_i$ be an augmenting cycle with $X_i \in X_f$ and $c_i$ positive. Let $m=\min \{ \frac{1}{c_i} c_f(X_i) \}$. Define a flow as $ \bar f(X)=f(X)+c_im$ if $X=X_i $ for some $i$, $\bar f(X)=f(X)-c_im $ if $X=-X_i$ for some i and $\bar f(X)=f(X) $ else. Now define $B=\{ X \in X^{(d)} : X=X_i \mbox{or} X=-X_i\} $.

$$
\sum_{ X \in X^{(d)}} \bar f(X) \partial(X)=\sum_{X \in X^{(d)}\backslash B} f(X) \partial(X) + \sum_{X \in B} \bar f(X) \partial (X) 
$$
$$
=\sum_{X \in X^{(d)}} f(X) \partial (X) + \sum_{X\in B} \pm c_im \partial(X) = m \partial \left( \sum_i c_i X_i \right )  = 0
$$

Capacity constraints hold by definition of $m$. As $X_i=T$ for some $i$ the value of the flow is strictly increased.

\end{proof}
\end{lem}

We would like to prove the converse of this lemma to devise a first algorithm for Higher Maxflow optimization. 

\begin{conj}\label{conjet} The converse of lemma (\ref{notmax}) holds in general.

\end{conj}

\begin{figure}[h!]
\begin{center}
\includegraphics[scale=0.7]{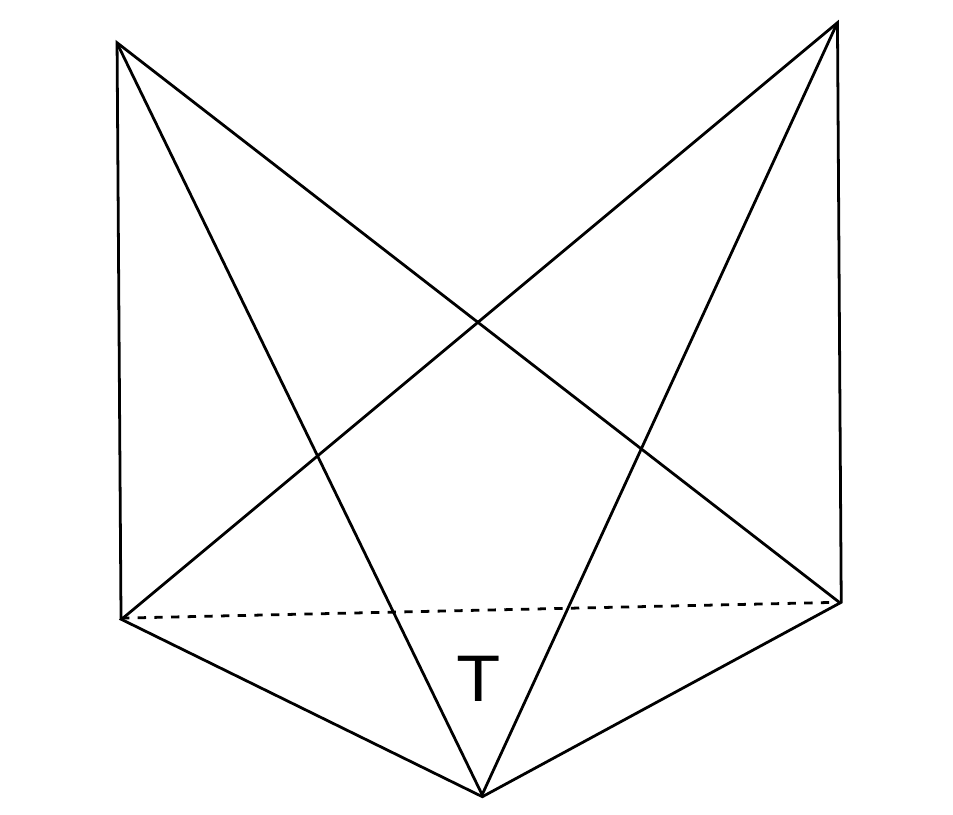} 
\caption{Oriented simplicial complex, the double tetrahedron $Sim_{2,2}$}
\end{center}
\end{figure} 

{\bf Example} Consider two tetrahedra oriented by the outward normal, with a common facet $T$, over the vertex set $\{ 1,2,3,4,5 \}$ We find the incidence matrix to be\\
$$
\left[ \begin{array}{ccccccc}
0 & 0 & 1 & 0 & 0 & 1 & -1 \\
-1 & 0 & 0 & -1 & 0 & 0 & 1 \\
0 & 0 & 0 & 1 & 0 & -1 & 0 \\
1 & 0 & -1 & 0 & 0 & 0 & 0 \\
0 & 1 & 0 & 0 & 1 & 0 & -1 \\
0 & 0 & 0 & 0 & -1 & 1 & 0 \\
0 & -1 & 1 & 0 & 0 & 0 & 0 \\
0 & 0 & 0 & -1 & 1 & 0 & 0 \\
-1 & 1 & 0 & 0 & 0 & 0 & 0 \\
\end{array} \right ]
$$

This matrix is also totally unimodular hence the problem has integer optima. In various cases we find this optima to be the sum of the minimum capacities over each tetrahedron. This optimum is achieved after finding two augmenting cycles corresponding to each tetrahedron. When the flow is maximum there is no augmenting cycle on the residual path. 

\begin{que}\label{uni} When is the incidence matrix of a $d$ dimensional network totally unimodular?

\end{que}

We will cite some important theorems that will give us some insight about the status of the question (\ref{uni}), and ultimately show that it is in fact false. We find certain family of networks where it holds.\\

The map $\partial_d$ has a unique matrix representation with respect of chosen basis. We denote such matrix as $[\partial_d]$. In fact it coincides with the incidence matrix we have defined. The kernel of $\partial_d$ is called the group of $p$-cycles and is denoted by $Z_d(X)$. The image of the map $\partial_{d+1}$, denoted by $B_d(X)$ is a subgroup of $C_d$ called the $d$-boundaries. We have that for any d $\partial_d \partial_{d+1}=0$ under composition so that $B_d \subset Z_d $. The $d$-dimensional homology group is defined as the quotient $H_d:=Z_d /B_d$. For a simplicial couplex $X$ we have that $Z_d(X), B_d(X), H_d(X)$ are finitely generated abelian groups and $H_d(X)$ depends only on the homotopy type of $X$.\\

For a subcomplex $X_0 \subset X$ we define the \emph{group of relative chains} of $X$ modulo $X_0$ as the quotient $C_d(X)/C_d(X_0):=C_d(X,X_0)$. The map $\partial_d$ induces a map $\partial_d^{(X,X_0)} : C_d(X,X_0) \rightarrow C_{d-1}(X,X_0)$ which also satisfies $\partial_d^{(X,X_0)} \partial_{d+1}^{(X,X_0)}=0 $  So we define the \emph{relative homology groups} as $H_d(X,X_0)=ker(\partial_d^{(X,X_0)}) / Im(\partial_{d+1}^{(X,X_0)}) := Z_d(X,X_0)/B_d(X,X_0)$. 

The following is a result by Dey-Hirani-Krishnamoorthy that gives a partial result to our question:\\

\begin{teo}\cite{Dey} For a finite simplicial complex triangulating a $d+1$ dimensional compact orientable manifold, $[\partial_{d+1}]$ is totally unimodular irrespective of the orientations of the simplices.

\end{teo}

The answer to question (\ref{uni}) is ``not always''. Consider the following counterexample:

{\bf Counterexample: }\cite{Dey} For certain simplicial complex triangulating the projective plane, the matrix $[\partial_2]$ is not totally unimodular.\\

This might not be exactly a counterexample for our conjecture as we need to prove existence of some facet that may work as source. However if we consider the two dimensional sphere positively oriented, and a triangulation of such manifold with consistent choice of orientations, then we may define any facet of such complex as the source, and then joining this complex to the triangulation of the projective plane by a vertex we find a submatrix of $[\partial_2]$ that is not totally unimodular.\\

The following is a theorem also due to Dey-Hirani-Krishnamoorthy that characterizes totally unimodular matrices arising from a boundary operator.

\begin{teo} \cite{Dey} \label{toed} $[\partial_d+1]$ is totally unimodular if and only if $H_d(Y,Y_0)$ is torsion free, for all pure subcomplexes $Y,Y_0$ of $X$ of dimensions $d+1$ and $d$ respectively, where $Y_0 \subset Y$.

\end{teo}

The following theorem yields another family of simplicial complexes where total unimodularity holds, namely, the family of $d$ dimensional complexes embeddable in $\R^d$.

\begin{teo}\cite{Dey} Let $K$ be a finite simplicial complex embedded in $\R^{d+1}$, then $H_d(L,L_0)$ is torsion free for all pure subcomplexes $L_0$ and $L$ of dimensions $d$ and $d+1$ respectively.

\end{teo}

It would be helpful to find more general families of simplicial complexes where unimodularity holds. Using theorem (\ref{toed}) we may give an alternative proof of the fact that for graphs, the incidence matrix is totally unimodular.

\begin{teo} For a directed graph $G$, the incidence matrix $[\partial_1]$ is totally unimodular. 
\begin{proof}
Suppose $G$ is connected. Let $L_0 \subset L$ be two subcomplexes of $G$ of dimension $0$ and $1$ respectively. We want to show that $H_0(L,L_0)$ is totally unimodular. Now $(L,L_0)$ is a \emph{good pair}, this implies that $H_0(L,L_0) \simeq H_0(L / L_0 ) $ which is a torsion free group isomorphic to $\mathbb{Z}^k$ where k is the number of connected components of $L/L_0$. The result follows by theorem (\ref{toed}).\\
\end{proof}

\end{teo}

\begin{defi}
Let $X$ be a simplicial complex. A pair of facets $(F,F')$ of $X$ is a leaf if for every face $H$ of $X$ we have $F \cap H \subset F'$. A simplicial tree is a simplicial connected complex $X$ such that every subset of facets of $X$ contains at least one leaf.
\end{defi}

\begin{lem}\label{inyec} Let $X$ be a pure simplicial complex of dimension 2, and suppose that for any pure sub complex of dimension 2 $L \subset X$, $H_1(L)=0$, then $[\partial_2]$ is totally unimodular.
\begin{proof} Let $L_0 \subset L $ be pure sub complexes of dimension 1 and 2 respectively. We have a long exact sequence in homology
$$
\ldots \rightarrow H_1(L) \rightarrow H_1(L,L_0) \rightarrow H_0(L_0)
$$
By exactness of the sequence we have that $H_1(L,L_0) \rightarrow H_0(L_0)$ is an invective map. As $H_0(L_0)$ is torsion free so is $H_1(L,L_0) $. The result follows by theorem (\ref{toed}).

\end{proof}

\end{lem}

\begin{lem} Let $X$ be a pure simplicial complex of dimension 2 such that there exist disjoint facets $\{T_1, \ldots, T_m\}$ such that $X \backslash \{ T_1, \ldots, T_m \} $ is a simplicial tree. Then for any pure subcomplex $L \subset X $ of dimension 2, $H_1(L)=0$.
\begin{proof} Let $L$ be a pure simplicial complex of dimension 2 with facets $\{F_1, \ldots, F_l \}$. Define $V$ as the pure simplicial complex with facets $\{F_1, \ldots, F_l\} \backslash \{T_1, \ldots, T_,\}$. By assumption, $V$ is a two dimensional subcomplex of a two dimensional simplicial tree so it is also a simplicial tree. Any simplicial tree is contractible and homology groups are homotopy invariant. Now for any $F$ facet of $L$ that is not in $V$ we have two cases: Every one dimensional face of $F$ is in $V$ or  not. In the first case, after contracting $V$, $F$ forms a sphere. In the second case $F$ has at least a one dimensional facet not in $V$ and $F$ can be contracted to $F \cap V$. Hence after contracting $V$ we see that $L$ is homotopy equivalent to a wedge of spheres so that $H_1(L)=0$.

\end{proof}

\end{lem}

\begin{cor} Let $X$ be a pure simplicial complex of dimension 2 such that there exist disjoint facets $\{T_1, \ldots, T_m\} $ such that $X \backslash \{ T_1, \ldots, T_m \} $ is a simplicial tree. Then $[\delta_2]$ is totally unimodular.

\end{cor}
Conjecture \ref{conjet} remains without answer.

\end{document}